\documentclass[review,onefignum,onetabnum]{siamart171218}

\usepackage{lipsum}
\usepackage{amsfonts}
\usepackage{graphicx}
\usepackage{epstopdf}
\usepackage{algorithmic}
\ifpdf
  \DeclareGraphicsExtensions{.eps,.pdf,.png,.jpg}
\else
  \DeclareGraphicsExtensions{.eps}
\fi

\newsiamremark{remark}{Remark}
\newsiamremark{hypothesis}{Hypothesis}
\crefname{hypothesis}{Hypothesis}{Hypotheses}
\newsiamthm{claim}{Claim}

\headers{Randomized Kaczmarz for Doubly-Noisy Linear Systems}{Bergou, Boucherouite, Dutta, Li, Ma}

\title{A Note on Randomized Kaczmarz Algorithm for Solving Doubly-Noisy Linear Systems}

\author{El Houcine Bergou \thanks{College of Computing, Mohammed VI Polytechnic University, Morocco, (\email{elhoucine.bergou@um6p.ma}).}
\and
Soumia Boucherouite \thanks{College of Computing, Mohammed VI Polytechnic University, Morocco, (\email{soumia.boucherouite@um6p.ma}).}
\and 
Aritra Dutta \thanks{Department of Mathematics, University of Central Florida, USA, (\email{aritra.dutta@ucf.edu}, \url{https://sciences.ucf.edu/math/person/aritra-dutta/}).}
\and
Xin Li\thanks{Department of Mathematics, University of Central Florida, USA, (\email{xin.li@ucf.edu}).} 
\and 
Anna Ma \thanks{Department of Mathematics, University of California, Irvine, USA, (\email{anna.ma@uci.edu}).}
}
\usepackage{amsopn}
\DeclareMathOperator{\diag}{diag}

\makeatletter
\newcommand*{\addFileDependency}[1]{
  \typeout{(#1)}
  \@addtofilelist{#1}
  \IfFileExists{#1}{}{\typeout{No file #1.}}
}
\makeatother

\newcommand{\E}[1]{{\mathbb{E}}\left[#1\right] }

\usepackage{caption}
\DeclareMathOperator{\argmin}{argmin}        

\newcommand{\R}{\mathbb{R}} 

\usepackage{mdframed} 
\usepackage{thmtools}

\usepackage{mdframed} 
\newcommand{\myNum}[1]{(\emph{#1})}

\usepackage{soul}

\usepackage{etoolbox}
\usepackage{pgfplotstable}
\usepackage{pgfplots}
\usepackage{booktabs}

\begin{document}

\maketitle

\begin{abstract}
 Large-scale linear systems, $Ax=b$, frequently arise in practice and demand effective iterative solvers. Often, these systems are noisy due to operational errors or faulty data-collection processes. In the past decade, the randomized Kaczmarz (RK) algorithm has been studied extensively as an efficient iterative solver for such systems. However, the convergence study of RK in the noisy regime is limited and considers measurement noise in the right-hand side vector, $b$. Unfortunately, in practice, that is not always the case; the coefficient matrix $A$ can also be noisy. In this paper, we analyze the convergence of RK for {\textit{doubly-noisy} linear systems, i.e., when the coefficient matrix, $A$, has additive or multiplicative noise, and $b$ is also noisy}. In our analyses, the quantity $\tilde R=\| \tilde A^{\dagger} \|^2 \|\tilde A \|_F^2$ influences the convergence of RK, where $\tilde A$ represents a noisy version of $A$. We claim that our analysis is robust and realistically applicable, as we do not require information about the noiseless coefficient matrix, $A$, and considering different conditions on noise, we can control the convergence of RK. {We perform numerical experiments to substantiate our theoretical findings.}
\end{abstract}




\section{Introduction}

In the digitized era, large-scale linear systems occur frequently in different research fields, including parameters estimation and inverse problems (e.g., the coefficient matrices in computed tomography applications can have $\sim10^9$ entries; see \cite{tomography}) \cite{ATarantola_2005, allgower1999nonlinear}, numerical solutions for PDEs (e.g., the most extensive numerical simulation of a turbulent fluid uses about $5\times 10^{10}$ grid points) \cite{doi:10.1073/pnas.1814058116, HAntil_DPKouri_MDLacasse_DRidzal_2016EDS}, 
generalized regression analysis in computer vision on high-resolution images or videos (characterized by a large number of rows, in the order of $\sim 10^5$ and more; see \cite{dutta2019online, dutta2017weighted}), meteorological predictions \cite{EBergou_SGratton_LNVicente_2016, asch2016, YTremolet_2007}, and many more. However, dimensionality is not the only curse of these complex linear systems; they are also poorly conditioned, which demands efficient solution strategies \cite{EBergou_2014, benzi2005numerical}.~Moreover, the data-curation process adds noise to the original data due to diverse operational errors. E.g., seismic and micro-seismic data collected using passive seismic techniques often differ from the true ones and lead to misinterpretation \cite{gajek2021errors}. Furthermore, stochastic differential equations (SDEs) are increasingly used and adopted to analyze many physical phenomena \cite{machiels1998numerical,keller1964stochastic}, and noise may enter the system due to ``internal degrees of freedom or as random variations of some external control parameters" \cite{du2002numerical}. As another example, in oil and gas industries, large saddle-point systems that arise from the numerical discretization of certain PDEs (e.g., a mixed-finite element discretization of a Poisson problem on a series of triangular meshes \cite{benzi2005numerical, chiu2014efficient}) can inherently {have} additive or multiplicative noise \cite{tambue2016weak,du2002numerical}. 

Instead of {the} noiseless linear system 
\begin{equation}\label{eq:mainPbLinearfreeerror}
 Ax = b,
\end{equation}
it is, therefore, more realistic to analyze {perturbed or, more generally, doubly-noisy\footnote{Here, we use the term ``doubly-noisy" to distinguish from the classical ``noisy" linear systems, where only the right-hand side vector, $b$, is noisy.} linear systems}. That is, 
we assume to have access to the noisy versions of \textit{both} $A$ and $b$, denoted as $\tilde A\in\R^{m\times n}$ and $\tilde b\in\R^{m}$, respectively. Note that the doubly-noisy linear system 
\begin{equation}\label{eq:mainPbLinear}
 \tilde A x \approx \tilde b,
\end{equation}
is not necessarily consistent. We assume 
that there is an underlying consistent noiseless linear system (\ref{eq:mainPbLinearfreeerror}).

The Kaczmarz method \cite{Kaczmarz1937} is a popular algorithm for solving {consistent} systems of linear equations.~{However, it has also been applied to approximate solutions for inconsistent linear systems \cite{needell2010randomized}}.~When applied to (\ref{eq:mainPbLinear}), it proceeds iteratively as follows
\begin{eqnarray}\label{eq:ka}
x_{k+1}=x_k-\frac{\tilde a_{i(k)}^\top x_k-\tilde b_{i(k)}}{\|\tilde a_{i(k)}\|^2}\tilde a_{i(k)},
\end{eqnarray}
where $k \ge 0$ is the iteration counter, $\tilde a_{i}^\top$ is the $i^{\rm th}$ row of the matrix $\tilde A$, $\tilde b_{i}$ is the $i^{\rm th}$ element of the vector, $\tilde b$. If $i(k)$ is chosen randomly with replacement based on some probability distribution, then \eqref{eq:ka} represents the {\em randomized Kaczmarz algorithm} (RK) \cite{Strohmer2013, Bai2018, sahu2020convergence}. 

In this paper, 
 we consider both the {\em additive noise}
 as well as {\em multiplicative noise}. Motivated by the multiplicative perturbation theory \cite{cai2011additive}, we recognize that certain structured least squares problems, e.g., those involving Vandermonde or Cauchy matrices \cite{drmac2020least}, are often highly ill-conditioned and have multiplicative backward errors \cite{castro2016multiplicative}. Therefore, we consider {\em multiplicative noise} that transforms the matrix, $A$ to $\tilde{A}=(I_m+E)A(I_n+F)\in\R^{m\times n}$, where {$E \in \R^{m\times m},$ $F \in \R^{n\times n}$,} $I_m$ and $I_n$ are identity matrices,~{and overall, $(I_m+E)$ and $(I_n+F)$ are nonsingular}. In both the additive and multiplicative error settings on $A$, we consider additive noise on the right-hand side vector $\tilde b=b+\epsilon,$~{where $\epsilon \in \R^{m}$ is the noise.} 

\subsection{Contributions and organization} The remainder of this paper is organized as follows. Section~\ref{sec:related_work} presents an overview of the related works. In Section~\ref{sec:perturb}, we present a simple analysis for the convergence of RK in the setting of additive noise on the coefficient matrix $A$ by using the perturbation bounds of the Moore–Penrose inverse \cite{wedin1973perturbation}. In Section~\ref{sec:direct}, we present an analysis of the convergence of iteration by directly analyzing RK on doubly-noisy linear systems. Finally, Section~\ref{sec:experiments} presents the numerical experiments which support our theoretical findings.

Our main contributions can be summarized as follows:
\begin{itemize}
    \item To the best of our knowledge, this work is the first effort to systematically analyze RK for doubly-noisy linear systems, i.e., when the coefficient matrix, $A$ {has} additive or multiplicative noise {and $b$ is a noisy measurement vector}; see Section \ref{sec:perturb} and Section~\ref{sec:direct}. 
    \item {While the convergence result of RK can be obtained using perturbation theory, it requires an assumption that the noise, $E$ does not change the rank of the original matrix ${A}$, that is, ${\rm rank}(A+E)=A,$ along with the consistency of the noisy system---they are restrictive and unrealistic assumptions.~Our most general result, Theorem \ref{theorem:general_theorem}, does not require such assumptions and holds for any general noise, $E$.}
    \item  In the setting in which the linear system only {has} right-hand side noise, our bounds produce comparable bounds to that in existing work; see Sections \ref{sec:related_work} { and \ref{sec:perturb}}, and Corollary~\ref{corr:general_theorem}.
    
 \item Finally, in Section \ref{sec:experiments}, we perform numerical experiments to validate our theoretical results. 
    
\end{itemize}

\def\buildTable#1{%
    \pgfplotstableread[col sep = comma]{#1}\rawdata%
    \pgfplotstabletypeset[before row= \midrule]\rawdata
}%

\begin{table}
\footnotesize
\centering
\addtolength{\tabcolsep}{-0.15em}
\begin{tabular}{c c c c c}  

 \midrule
 &    &  {Convergence} & {Convergence}  & \\
Quantity & Linear system   &  Rate & Horizon & Reference\\
 \midrule
$\mathbb{E}[\|x_k-x_{\rm LS}\|^2]$ & $Ax \approx b+\epsilon$ & $(1-\frac{1}{R})$ & $R \left(\max_i \frac{\epsilon_i}{\|A_i\|}\right)^2 $ & \cite{needell2010randomized}\\
\midrule
$\mathbb{E}[\|x_k-x_{\rm LS}\|^2]$ & $Ax\approx b+\epsilon$ & $(1-\frac{1}{R})$ & $\frac{\|\epsilon \|^2}{\sigma_{\rm min}^2(A)}$ & Theorem \ref{thm:zouzias}; \\
&&&& see \cite{zouzias2013randomized}\\
\midrule
$\mathbb{E}\|x_k-x_{\rm LS}\|$ & $(A+E)x\approx b+\epsilon$ & $(1-\frac{1}{\tilde R})^{\frac{1}{2}}$ & $\frac{2\|  x_{{\rm LS}} \|\|A^\dagger\|\|E\|}{1-\|A^\dagger\|\|E\|}$ & Theorem \ref{theorem:ls_rk_pnls}; \\
&&& $+ \frac{\|\epsilon\|}{\sigma_{\min}(\tilde{A})}$ & this paper\\
\midrule
$\mathbb{E}[\|x_k-x_{\rm LS}\|^2]$ & $(A+E)x\approx b+\epsilon$  & $(1-\frac{1}{\tilde{R}})$ & $\frac{\|Ex_{LS} - \epsilon \|^2}{\sigma_{\rm min}^2(A)}$& Theorem \ref{theorem:general_theorem},\\
& & & & this paper\\
\midrule
$\mathbb{E}[\|x_k-x_{\rm LS}\|^2]$ & $(I_m+E)A(I_n+F)x\approx b+\epsilon$  & $(1-\frac{1}{\tilde{R}})$ & $\frac{\|\Delta Ax_{LS} - {\epsilon}\|^2}{\sigma_{\rm min}^2(\tilde{A})}$ & Corollary \ref{theorem:general_theorem_multiplicative},\\ 
& & & & this paper\\
\midrule
\end{tabular}
\caption{Summary of RK applied to noisy linear systems. In each row, we assume that $Ax = b$ is a consistent linear system and we define $R=\| A^\dagger\|^2 \|A \|^2_F$ and $\tilde{R}=\| \tilde{A}^\dagger\|^2 \|\tilde{A}\|^2_F$. Note that, $\Delta A=(EA + AF + EAF).$ {For details about convergence rate and horizon, see the comment following Theorem~\ref{thm:zouzias}.}}\label{table1}
\end{table}

\subsection{Notation}\label{sec:notation}
{The $i^{\rm th}$ component of a vector $x$ is denoted as $x_i,$ $x^\top$ denotes the transpose of $x$, and $\| x \|$ denotes the $\ell_2$-norm of the vector $x$. For a matrix $M\in\R^{m\times n}$, the $i^{\rm th}$ singular value of $M$ is denoted by $\sigma_{i}(M)$, $\sigma_{\rm min}(M)$ is the smallest nonzero singular value, and $\sigma_{\rm max}(M)=\sigma_1$, is the largest singular value. The column space or the range of $M$ is defined by ${\rm range}(M).$~We denote the (Moore-Penrose)-pseudoinverse of $M$ by $M^\dagger.$~The spectral norm and the Frobenius norm of $M$ are denoted by  $\|M\|$ and $\|M\|_F,$ respectively.}
{For a linear system $Mx=y$, we denote $x_{\rm LS}:=M^\dagger y$. As it is well-known, $x_{\rm LS}\in\argmin_x \| Mx - y \|^2,$ and it is the least squares solution with the minimal norm.}

\subsection{Brief literature review}\label{sec:related_work}

The Kaczmarz algorithm (KA) \cite{Kaczmarz1937} was proposed by Stefan Kaczmarz, and Strohmer and Vershynin proposed and analyzed its randomized variant, RK \cite{Strohmer2013}. Due to its simplicity and low-memory footprint, many works have considered RK in different settings since then. E.g., in \cite{tan2019phase, jeong2017convergence}, the authors propose an RK-type algorithm for solving the phase retrieval problem. \cite{haddock2021greed, bai2018relaxed} study the interplay between RK using greedy and random row select. Many works have proposed variations of RK, including block-wise selection methods, e.g., \cite{needell2014paved, necoara2019faster, du2020randomized}; quantile-based methods for sparse corruptions in $b$, e.g., \cite{jarman2021quantilerk,haddock2022quantile,steinerberger2023quantile}; and variations for when $x$ is assumed to be sparse \cite{lorenz2014sparse, lei2018learning, schopfer2019linear}. 

In addition, previous works have also considered RK applied to {inconsistent,} noisy linear systems. The first such work was~\cite{needell2010randomized}, {which proved that the RK converges to a ball centered around the least squares solution. We refer to the radius of this ball as the \textit{convergence horizon}; also, see \cite{needell2014paved}}. 
When $E = 0$ and $\epsilon$ has no structural assumptions, Needell, {in ~\cite{needell2010randomized},} proved that the iterates of RK converge in expectation to $x_{\rm LS}$ within a convergence horizon of $R \gamma ^2 $, where $R=\| A^\dagger\|^2 \|A \|^2_F$ is a scaled condition number of $A$ and $\gamma = \max_i \epsilon_i / \| a_i\|_2$. Needell and Tropp later improved this convergence horizon for standardized matrices (row-normalized) to $\| \epsilon\|^2\| A^\dagger \|^2$~\cite{needell2014paved}.~Jarman and Needell in~\cite{jarman2021quantilerk} proposed quantile-RK for the setting $E=0$ and $\tilde{b} = b + \epsilon,$ where $\epsilon$ is a sparse corruption vector. The result was further generalized by Zouzias and Freris~\cite{zouzias2013randomized} who showed a convergence horizon of $\|\epsilon \|^2 \| A^\dagger \|^2$ for any matrix $A$. In the same work, Zouzias and Freris proposed the randomized extended Kaczmarz algorithm (REK) which was introduced to handle noise in the right-hand side vector $\tilde{b}$, and has been shown to converge to the least squares solution for noisy linear systems~\cite{zouzias2013randomized, du2019tight}.~Note that the REK algorithm, and its variations~\cite{bai2019partially, wu2022two, du2020randomized}, require not only rows of the matrix $A$, but columns of $A$ as well, which may not be feasible in all applications. We refer to Table~\ref{table1} for an overview of these results.

\section{Perturbation Analysis}\label{sec:perturb}
The Moore–Penrose inverse or generalized inverse \cite{wedin1973perturbation} is well-studied and plays a crucial role in matrix theory and numerical analysis. We can use the perturbation bounds of the Moore–Penrose inverse \cite{wedin1973perturbation} of matrix $A$ to quickly arrive at some convergence bounds for RK for doubly-noisy systems but under more restrictive conditions than one may desire. In this section, we 
show how far perturbation analysis 
could take us in convergence analysis.
In contrast to our main result in Section~\ref{sec:direct}, this analysis will require stronger assumptions on the noisy matrix $E$ and, under certain conditions, obtain weaker convergence bounds than we can get by directly analyzing RK on doubly-noisy linear systems.

We start with the convergence result of RK proved by 
Strohmer and Vershynin in \cite[Theorem 2]{Strohmer2013} for the noiseless system, $Ax=b$, and by
Zouzias and Freris in \cite[Theorem 3.7]{zouzias2013randomized} for the noisy system, $A x\approx \tilde b$. Let $R = \| A^{\dagger} \|^2 \|A \|_F^2$.
\begin{theorem}\label{thm:zouzias}{\rm (\cite{Strohmer2013,zouzias2013randomized})} Set $x_0$ to be any vector in the row space of $A$, that is, $x_0\in {\rm range}(A^\top)$. 
\begin{enumerate}
    \item[(I)] Let $x_k$ be the iterate of RK applied to the noiseless system $Ax = b$. Assume that $A x= b$ is consistent and let $x_{\rm LS}=A^\dagger b$. Then
\begin{equation} \label{eq:strohmer}
 \mathbb{E}\| x_{k} - x_{{\rm LS}}\|^2 \leq \left(1-\frac{1}{R}\right)^{k} \| x_0 - x_{{\rm LS}} \|^2.  
\end{equation}
\item[(II)] Let $x_k$ be the iterate of RK applied to the noisy system $Ax \approx b+\epsilon$, where we assume that $A x= b$ is consistent with $x_{\rm LS}=A^\dagger b$. Then
\begin{equation}\label{eq:zouzias}
 \mathbb{E}\| x_{k} - x_{{\rm LS}}\|^2 \leq \left(1-\frac{1}{R}\right)^{k} \| x_0 - x_{{\rm LS}} \|^2 + \frac{\|\epsilon\|^2}{\sigma^2_{\rm min}(A)}.  
\end{equation}
\end{enumerate}
\end{theorem}

{On the right-hand side of inequality (\ref{eq:zouzias}), the term $(1-\frac{1}{R})$ denotes the convergence rate, and the second term in the sum is the convergence horizon.}

In what follows, we attempt to obtain convergence bounds for RK applied to the doubly-noisy system using Theorem~\ref{thm:zouzias}, properties of the generalized inverse,  and the perturbation theory, so that the bounds are in terms of $\tilde R = \| \tilde A^{\dagger} \|^2 \|\tilde A \|_F^2$. 

First, we recall a classic, intermediate result from \cite{castro2016multiplicative, bjorck1996numerical} to quantify the difference between the least squares solution of the noiseless system, $Ax=b$, and the noisy system, $(A+E)x=b+\epsilon$. 
\begin{lemma}\label{lemma:perturb_ls} \rm (c.f.  \cite[Theorem 1.4.6]{bjorck1996numerical}, \cite{castro2016multiplicative})
Let ${\rm rank}(A)={\rm rank}(A+E)$ and $\|A^\dagger\|\|E\| < 1$. Let $Ax=b$ be consistent. Denote $x_{\rm LS}:=A^\dagger b$ and $x_{\rm NLS}:=(A+E)^\dagger(b+\epsilon).$ Then
$$\| x_{{\rm NLS}} - x_{{\rm LS}} \| \leq \frac{\|  x_{{\rm LS}} \|}{1-\|A^\dagger\|\|E\|}\left(2\|A^\dagger\|\|E\|+\frac{\|A^\dagger\|\|\epsilon\|}{\| x_{{\rm LS}} \|} \right).$$
\end{lemma}



Now, we can quantify the difference between $x_{{\rm LS}}$ and the iterates generated by \eqref{eq:ka} 
as follows. 
\begin{theorem}\label{theorem:ls_rk_addtivenoise}
Let ${\rm rank}(A)={\rm rank}(A+E)$ and $\|A^\dagger\|\|E\| < 1$. Let ${x}_{k}$ be defined in \eqref{eq:ka}. Suppose the doubly-noisy linear system, $\Tilde{A}x=\Tilde{b}$, and the noiseless system, $Ax=b$, are consistent. Then
\begin{align*}
 \textstyle{\mathbb{E} \|{x}_{k} - x_{{\rm LS}} \| \leq \left(1-\frac{1}{\tilde R}\right)^{k/2} \| x_0 - x_{{\rm NLS}} \| + \frac{\|  x_{{\rm LS}} \|}{1-\|A^\dagger\|\|E\|}\left(2\|A^\dagger\|\|E\|+\frac{\|A^\dagger\|\|\epsilon\|}{\| x_{{\rm LS}} \|} \right),}
    \end{align*}
    where $\tilde R = \| \tilde A^{\dagger} \|^2 \|\tilde A \|_F^2$,  $x_{\rm LS}=A^\dagger b$, and $x_{\rm NLS}=(A+E)^\dagger(b+\epsilon).$
\end{theorem}
\begin{proof}
{Applying the bound from (I) in Theorem~\ref{thm:zouzias} to the noisy system $\Tilde{A}x=\Tilde{b}$, which is consistent by the assumption, we obtain
\begin{eqnarray}\label{eq:JE0}
    \mathbb{E}\| x_{k} - x_{{\rm NLS}}\|^2 \leq \left(1-\frac{1}{\tilde R}\right)^{k} \| x_0 - x_{{\rm NLS}} \|^2.
\end{eqnarray}
Using Jensen's inequality we have $ \mathbb{E}\|x_k-x_{\rm NLS}\|\leq
\left(\mathbb{E}\|x_k-x_{\rm NLS}\|^2\right)^{1/2}$, which, together with (\ref{eq:JE0}), yields
\begin{eqnarray}\label{eq:JE}
    \mathbb{E}\| x_{k} - x_{{\rm NLS}}\| \leq \left(1-\frac{1}{\tilde R}\right)^{k/2} \| x_0 - x_{{\rm NLS}} \|.
\end{eqnarray}
Thus}
\begin{align*}
   \mathbb{E} \|{x}_{k} - x_{{\rm LS}} \| 
    &\leq \mathbb{E} \| {x}_{k}  - x_{{\rm NLS}} \| + \| x_{{\rm NLS}} - x_{{\rm LS}} \|\\
    & \leq \left(1-\frac{1}{\tilde R}\right)^{k/2} \| x_0 - x_{{\rm NLS}} \| + \| x_{{\rm NLS}} - x_{{\rm LS}} \|.
\end{align*}
Bounding $\| x_{{\rm NLS}} - x_{{\rm LS}} \|$ by using Lemma \ref{lemma:perturb_ls} obtains the desired result.
\end{proof}

We note that for the particular case when $E=0$, we get a weaker bound compared to RK applied to the consistent system, $ Ax=\tilde b$ in Theorem~\ref{thm:zouzias} (I). This observation suggests that combining perturbation theory with triangle inequality will result in losing the sharpness of the bounds, and an alternative analysis for stronger bounds is required. We provide such analysis in Section \ref{sec:direct}.


Further, one can quantify the difference between $x_{{\rm LS}}$ and the iterates generated by \eqref{eq:ka}, via considering the intermediate solution, $x_{{\rm PNLS}}$ which is the least squares solution to the {\em partially noisy} linear system , $(A+E)x\approx b$. We start with a bound on the error between the solution to the noiseless system and the least squares solution to the partially noisy system:

\begin{theorem}\label{theorem:ls_rk_pnls}
Let ${\rm rank}(A)={\rm rank}(A+E)$ and $\|A^\dagger\|\|E\| < 1$. Let ${x}_{k+1}$ be defined in \eqref{eq:ka}. Suppose the partially noisy linear system, $\tilde A x=b$,  and the noiseless system, $Ax=b$, are consistent. Then
\begin{align*}
   \mathbb{E} \|{x}_{k}  - x_{{\rm LS}} \| &\leq \left(1-\frac{1}{ \tilde R}\right)^{k/2} \| x_0  - x_{{\rm PNLS}}\| + \frac{2\|  x_{{\rm LS}} \|\|A^\dagger\|\|E\|}{1-\|A^\dagger\|\|E\|} + \frac{\|\epsilon\|}{\sigma_{\min}(\tilde{A})},
    \end{align*}
    where $ \tilde R = \| \tilde A^{\dagger} \|^2 \|\tilde A \|_F^2$,  $x_{\rm LS}=A^\dagger b$, and $x_{\rm PNLS}:=(A+E)^\dagger b.$
\end{theorem}
\begin{proof}
The result can be obtained by using the triangle inequality and then using Theorem \ref{thm:zouzias} and {Lemma \ref{lemma:perturb_ls} with $\epsilon = 0$.}
\end{proof}

\section{Main Results}\label{sec:direct}
We obtained the results in Section~\ref{sec:perturb} by using the known estimate of the least squares solutions of the doubly-noisy and the noiseless systems. Although we arrived at some interesting bounds, Theorem \ref{theorem:ls_rk_addtivenoise} and \ref{theorem:ls_rk_pnls} are restrictive due to two reasons: \myNum{i} {In both theorems, we require the noisy systems, doubly or partial, to be consistent, which may not be realistic and too restrictive}. \myNum{ii} In both theorems, we require the matrices, $A$ and $A+E$ to have the same rank, i.e., $E$ cannot change the rank of $A$, and connected to the original matrix $A$ via $\|A^\dagger\|\|E\| < 1$. The rank change may be a restrictive condition, especially if $A$ is only approximately low rank.
Therefore, a natural question is: {\em Can we avoid these issues altogether?} Theorem \ref{theorem:general_theorem} answers this. 

The result in Theorem \ref{theorem:general_theorem} directly connects the iterates, $x_k$, of RK applied to the doubly-noisy linear system, $\Tilde{A}x \approx \Tilde{b}$ to $x_{\rm LS}$ without going through either $x_{\rm NLS}$ or $x_{\rm PNLS}$. Thus, it will not need the consistency of the noisy system. We only require {\em consistency} of the original noiseless system, $Ax=b$, which is a mild condition. Moreover, we dispense any restriction on the noise term, that is, $E$ does not need to be acute, and $\|A^\dagger\|\|E\| < 1$ does not need to hold. An immediate consequence is that we do not have to worry about the rank change of $\tilde{A}$ due to noisy perturbation, $E$.

\begin{theorem} \label{theorem:general_theorem} Let $x_k$ be the iterate of RK applied to the doubly-noisy linear system, $\Tilde{A}x \approx \Tilde{b}$, with {$x_0-x_{\rm LS}\in {\rm range}(\tilde{A}^\top)$}. Assume the noiseless system $Ax=b$ is consistent with $x_{\rm LS}=A^\dagger b$. Then
\begin{equation}
 \mathbb{E}\| x_{k} - x_{{\rm LS}}\|^2 \leq \left(1-\frac{1}{\Tilde{R}}\right)^{k} \| x_0 - x_{{\rm LS}} \|^2 + \frac{\| E x_{\rm LS} - \epsilon \|^2}{\sigma_{\rm min}^2(\Tilde{A})},  
\end{equation}
where 
$\tilde R = \| \tilde A^{\dagger} \|^2 \|\tilde A \|_F^2$.
\end{theorem}

\begin{proof}
We start with (\ref{eq:ka}) and write
\begin{align*}
 x_{k+1} & = x_k-\frac{\tilde{a}^\top_{i(k)}x_k-\tilde{b}_{i(k)}}{\|\tilde{a}_{i(k)}\|^2}\tilde{a}_{i(k)} \\
& = x_k-\frac{\tilde{a}^\top_{i(k)}(x_k-x_{\rm LS})+E_{i(k)}x_{\rm LS}-\epsilon_{i(k)}}{\|\tilde{a}_{i(k)}\|^2}\tilde{a}_{i(k)}.
\end{align*}
Thus, we have
\begin{eqnarray*}
x_{k+1}-x_{\rm LS}
&=&x_k-x_{\rm LS}-\frac{\tilde{a}^\top_{i(k)}(x_k-x_{\rm LS})+E_{i(k)}x_{\rm LS}-\epsilon_{i(k)}}{\|\tilde{a}_{i(k)}\|^2}\tilde{a}_{i(k)}\\
&=&
\left(I-\frac{\tilde{a}_{i(k)}\tilde{a}^\top_{i(k)}}{\|\tilde{a}_{i(k)}\|^2}\right)(x_k-x_{\rm LS}){-}\frac{E_{i(k)}x_{\rm LS}-\epsilon_{i(k)}}{\|\tilde{a}_{i(k)}\|^2}\tilde{a}_{i(k)}.  
\end{eqnarray*}
Note that the two terms, $\left(I-\frac{\tilde{a}_{i(k)}\tilde{a}^\top_{i(k)}}{\|\tilde{a}_{i(k)}\|^2}\right)(x_k-x_{\rm LS})$ and $\frac{E_{i(k)}x_{\rm LS}-\epsilon_{i(k)}}{\|\tilde{a}_{i(k)}\|^2}\tilde{a}_{i(k)}$ are orthogonal; moreover, the matrix, $\left(I-\frac{\tilde{a}_{i(k)}\tilde{a}^\top_{i(k)}}{\|\tilde{a}_{i(k)}\|^2}\right)$ is a projection matrix.
That is, $\left(I-\frac{\tilde{a}_{i(k)}\tilde{a}^\top_{i(k)}}{\|\tilde{a}_{i(k)}\|^2}\right)^2=I-\frac{\tilde{a}_{i(k)}\tilde{a}^\top_{i(k)}}{\|\tilde{a}_{i(k)}\|^2}$. Hence, we obtain
\begin{equation*}
    \|x_{k+1}-x_{\rm LS}\|^2=
\left\|\left(I-\frac{\tilde{a}_{i(k)}\tilde{a}^\top_{i(k)}}{\|\tilde{a}_{i(k)}\|^2}\right)(x_k-x_{\rm LS})\right\|^2+\left\|\frac{E_{i(k)}x_{\rm LS}-\epsilon_{i(k)}}{\|\tilde{a}_{i(k)}\|^2}\tilde{a}_{i(k)}\right\|^2,
\end{equation*}
which can be simplified to
\begin{equation*}
    \|x_{k+1}-x_{\rm LS}\|^2=
(x_k-x_{\rm LS})^\top\left(I-\frac{\tilde{a}_{i(k)}\tilde{a}^\top_{i(k)}}{\|\tilde{a}_{i(k)}\|^2}\right)(x_k-x_{\rm LS})+\frac{(E_{i(k)}x_{\rm LS}-\epsilon_{i(k)})^2}{\|\tilde{a}_{i(k)}\|^2}.
\end{equation*}
Now, taking expectation on $i(k)$ conditioning with $x_k$ to get
\begin{equation*}
  \E{\|x_{k+1}-x_{\rm LS}\|^2|x_k}=
(x_k-x_{\rm LS})^\top\left(I-\frac{\tilde{A}^\top\tilde{A}}{\|\tilde{A}\|_F^2}\right)(x_k-x_{\rm LS})+\frac{\|Ex_{\rm LS}-\epsilon\|^2}{\|\tilde{A}\|_F^2}.  
\end{equation*}

Note that, for calculating expectation in the last term, {the probability of sampling row $i$ is proportional to the row norm of the given noisy matrix $\tilde{A}$, i.e., $p_{i(k)}=\frac{\|\tilde{a}_{i(k)}\|^2}{\sum_{i(k)=1}^m\|\tilde{a}_{i(k)}\|^2}=\frac{\|\tilde{a}_{i(k)}\|^2}{\|\tilde{A}\|^2}$.}
This sampling rule is similar to that of Strohmer and Vershynin \cite{Strohmer2013}. 

Since $x_0-x_{\rm LS}\in {\rm range}(\tilde{A}^\top)$, which would imply $x_k-x_{\rm LS}\in {\rm range}(\tilde{A}^\top)$, the first term could be bounded by 
\begin{equation*}
   \left(1-\frac{{\sigma}_{\rm min}^2(\tilde{A})}{\|\tilde{A}\|_F^2}\right)\|x_k-x_{\rm LS}\|^2. 
\end{equation*}

Repeatedly using the inequality completes the proof.
\end{proof}

\begin{remark} In our analysis, the quantity, $\tilde R=\| \tilde A^{\dagger} \|^2 \|\tilde A \|_F^2$, influences the convergence of noisy RK, not $R=\|A^{\dagger} \|^2 \|A \|_F^2.$ This is expected as the convergence rate of RK applied to doubly-noisy linear systems will depend on the scaled conditioning of the noisy system, not the noiseless system since RK has no way to distinguish between $A$ and $\tilde{A}$ within its iterates.
\end{remark}

We need an additional proposition to compare the convergence horizon of Theorem~\ref{theorem:general_theorem} and Theorem~\ref{theorem:ls_rk_pnls}. Weyl's inequality \cite{horn2012matrix} gives the perturbation bound between a Hermitian matrix and its perturbed form in terms of their eigenvalues. However, for any general matrix, the following singular value perturbation result can be obtained from Weyl’s inequalities for eigenvalues of Hermitian matrices. 

\begin{proposition}\cite[Corollary 7.3.5.]{horn2012matrix}\label{prop:singular_value_perturb}
Let $A,E\in\mathbb{C}^{m\times n}$ and $r=\min\{m,n\}.$ Let $\sigma_1(A)\ge\cdots\ge\sigma_r(A),$
and $\sigma_1(A+E) \ge \cdots\ge \sigma_r(A+E)$ be the nonincreasingly ordered singular values of $A$ and $A+E$, respectively. Then $|\sigma_i(A+E)-\sigma_i(A)|\le \| |E|\|$ for $1\le i\le r.$
\end{proposition}
For our case, $A,E\in\mathbb{R}^{m\times n}$. From the above proposition, if ${\rm rank}(A)={\rm rank}(A+E)$, for $1\le i\le r$, we have $-\|E\|\le \sigma_i(A+E)-\sigma_i(A)\le \|E\|,$ which gives $  {\sigma}_{\rm min}({A}) - \|E \| \leq {\sigma}_{\rm min}(\tilde{A}) < 2{\sigma}_{\rm min}(\tilde{A}).$ 
Therefore, under this scenario, we can compare the convergence horizon of Theorem~\ref{theorem:general_theorem} and Theorem~\ref{theorem:ls_rk_pnls}. 

\begin{remark}\label{rem:better_horizon}
The convergence horizon resulting from Theorem~\ref{theorem:general_theorem} is (without considering the squares)\footnote{We can get directly the unsquared form of the inequality by using Jensen's inequality. }: 
\begin{equation*}
    \frac{\| Ex_{\rm LS} - \epsilon\|}{{\sigma}_{\rm min}(\tilde{A})}.
\end{equation*}
If 
$2{\sigma}_{\rm min}(\tilde{A}) > {\sigma}_{\rm min}({A}) - \|E \|$, we can arrive at a better convergence horizon in Theorem~\ref{theorem:general_theorem}  than that of Theorem~\ref{theorem:ls_rk_pnls}. Comparing the bounds we have that the horizon in Theorem~\ref{theorem:ls_rk_pnls} is:
\begin{equation*}
    \frac{2 \|x_{\rm LS} \| \| E \|}{ {\sigma}_{\rm min}({A}) - \| E \| } + \frac{\| \epsilon \|}{{\sigma}_{\rm min}(\tilde{A})}  
\end{equation*}
Thus, we have that:
\begin{align*}
    \frac{\| Ex_{\rm LS} - \epsilon\|}{{\sigma}_{\rm min}(\tilde{A})} \leq \frac{ \| E \| \|x_{\rm LS}\|}{{\sigma}_{\rm min}(\tilde{A})} + \frac{\| \epsilon\|}{{\sigma}_{\rm min}(\tilde{A})} \leq \frac{2 \| E \| \|x_{\rm LS}\|}{{\sigma}_{\rm min}({A}) - \| E \| } + \frac{\| \epsilon\|}{{\sigma}_{\rm min}({\Tilde{A}})}. 
\end{align*}
\end{remark}

When there is no noise in the coefficient matrix, that is, $E=0$, Theorem \ref{theorem:general_theorem} obtains the same convergence bounds as Theorem 3.7 in \cite{zouzias2013randomized}; see Theorem \ref{thm:zouzias} (II) in this work.
\begin{corollary}\label{corr:general_theorem} Set $E=0$. Let $x_0\in {\rm range}({A}^\top)$ and assume the noiseless system $Ax=b$ is consistent. Then
\begin{equation}
 \mathbb{E}\| x_{k} - x_{{\rm LS}}\|^2 \leq \left(1-\frac{1}{{R}}\right)^{k} \| x_0 - x_{{\rm LS}} \|^2 + \frac{\|\epsilon \|^2}{\sigma_{\rm min}^2({A})}, 
\end{equation}
where 
${R} = \|{A}^{\dagger} \|^2 \|{A}\|_F^2$.  
\end{corollary}

\paragraph{Multiplicative Noise} As a consequence of Theorem~\ref{theorem:general_theorem}, we can immediately obtain bounds for multiplicative noise. 
For multiplicative noise, we consider the following linear system 
\begin{eqnarray}\label{eq:multiplicative_noise}
(I_m+E)A(I_n+F)x \approx b+ {\epsilon},  
\end{eqnarray}
where $(I_m+E)$ and $(I_n+F)$ are nonsingular. 
In this case, we cast: $\Tilde{A}  = (I_m+E)A(I_n+F) = A + EA + AF + EAF = A + \Delta A,$ where $\Delta A = EA + AF + EAF$ and $\tilde{b}=b+\epsilon$.

\begin{corollary}{(Multiplicative Noise)}\label{theorem:general_theorem_multiplicative} Let $x_k$ be the iterate of RK applied to the doubly-noisy linear system, $\Tilde{A}x \approx \Tilde{b}$, with $x_0-x_{\rm LS}\in {\rm range}(\tilde{A}^\top)$ and $\Tilde{A}  = (I_m+E)A(I_n+F) = A + EA + AF + EAF = A + \Delta A,$ such that $(I_m+E)$ and $(I_n+F)$ are nonsingular. Assume the noiseless system $Ax=b$ is consistent. Then
\begin{equation}
 \mathbb{E}\| x_{k} - x_{{\rm LS}}\|^2 \leq \left(1-\frac{1}{\Tilde{R}}\right)^{k} \| x_0 - x_{{\rm LS}} \|^2 + \frac{\| \Delta A x_{\rm LS} - \epsilon \|^2}{\sigma_{\rm min}^2(\Tilde{A})},  
\end{equation}
where 
$\tilde R = \| \tilde A^{\dagger} \|^2 \|\tilde A \|_F^2$.
   \end{corollary} 

\begin{remark}
Similar to the additive noise case, by using perturbation analysis, we can quantify the difference between $x_{{\rm LS}}$ and the iterates generated by \eqref{eq:ka} via $x_{{\rm NLS}}$. We can use Theorem 4.1 from \cite{castro2016multiplicative} that quantifies the difference between $x_{\rm NLS}$ and $x_{\rm LS}$ for multiplicative noise and is valid for perturbation of any size. If the noisy linear system, \eqref{eq:mainPbLinear} is consistent, then we have
\begin{align*}
    \mathbb{E} \|{x}_{k+1} - x_{{\rm LS}} \| &\leq \left(1-\frac{1}{\tilde R}\right)^{k/2} \| x_0 - x_{{\rm NLS}} \| + e_1 \|x_{{\rm LS}} \| + e_2 \|A^\dagger\|\|b\|,
\end{align*}
    where $e_1 = \sqrt{\|F\|^2+\|(I+F)^{-1}F\|^2}$ and \\
    $\textstyle{e_2 = (1+e_1)(\frac{\epsilon}{\|b\|}+(1+\frac{\epsilon}{\|b\|})\sqrt{\|E\|^2+\|(I+E)^{-1}E\|^2})}$.
\end{remark}

\subsection{Using noise to speed up RK}\label{sec:preconditioner} 
    
Suppose we are interested in finding a solution to the noiseless system, $Ax = b$ (assuming that $A$ and $b$ are known) as in \eqref{eq:mainPbLinearfreeerror}. In this case, we know that the convergence rate for RK in solving \eqref{eq:mainPbLinearfreeerror} is $ \rho = 1-\frac{1}{R}$. However, Theorem \ref{theorem:general_theorem} tells us, for noisy systems as in \eqref{eq:mainPbLinear}, the convergence rate for RK in solving \eqref{eq:mainPbLinear} is $\tilde \rho = 1-\frac{1}{\tilde R}.$ This leads us to an interesting fact: {\em If we have a rough approximation of the singular values, then one can inject noise and speed up the convergence at the beginning of the optimization and then return back to the exact system after some iterations.} In other words, we can incorporate {\em additive pre-conditioners} to speed up the convergence of our iterative method. 
In what follows, we illustrate the above hypothesis with a simple example.

\paragraph{A proof-of-concept example} 
Let $A\in\R^{m\times n}$ be given matrix of rank $r.$ Let $A=U \Sigma V^\top$ be its SVD, where $U\in\R^{m\times r}$ and $V\in\R^{n\times r}$ are column orthonormal matrices and $\Sigma=\diag(\sigma_1,~\sigma_2,\cdots,\sigma_r)\in\R^{r \times r}$ be a diagonal matrix with nonincreasingly ordered singular values of $A$ along the diagonal.~We can write $A=\sum_{i=1}^r\sigma_iu_iv_i^\top$, where $u_i$ and $v_i$ are left and right singular vectors, respectively. Set $E=(\sigma_{r-1}-\sigma_r)u_r v_r^{\top}$, and therefore, $\Tilde{A} = A + E.$ 
Based on this, we have $R = \frac{\Sigma_{i=1}^r \sigma_i^2}{\sigma_r^2}$, $\Tilde{R} = \frac{\Sigma_{i=1}^{r-1} \sigma_i^2 + \sigma_{r-1}^2}{\sigma_{r-1}^2}$, $\sigma_{\min}^2(\Tilde{A}) = \sigma_{r-1}^2 $, 
and $\|E\|^2 = (\sigma_{r-1}-\sigma_r)^2.$

On one hand, if RK is applied on a consistent noiseless system, $Ax = b$, we have: 
\begin{equation*}
    \mathbb{E}\| x_{k} - x_{{\rm LS}}\|^2 \leq \left(1-\frac{1}{R}\right)^{k} \| x_0 - x_{{\rm LS}} \|^2.
\end{equation*}
Hence, for the expected squared approximation error to reach a tolerance $\tau> 0 $, we need 
\begin{equation*}
    K \geq \frac{ \log \left(\frac{\tau}{\| x_0 - x_{{\rm LS}} \|^2}\right)}{\log\left(1-1/R\right)}
\end{equation*}
iterations. On the other hand, if RK is applied to the noisy system, $\Tilde{A}x \approx b$, from Theorem~\ref{theorem:general_theorem} we have:
\begin{equation*}
    \mathbb{E}\| x_{k} - x_{{\rm LS}}\|^2 \leq \left(1-\frac{1}{\Tilde{R}}\right)^{k} \| x_0 - x_{{\rm LS}} \|^2 + \frac{\| E x_{\rm LS}\|^2}{\sigma_{r-1}^2(\Tilde{A})}.
\end{equation*}
Letting $\tau_0 = \frac{\| E x_{\rm LS} - \epsilon \|^2}{\sigma_{r-1}^2(\Tilde{A})}$ (the convergence horizon from Theorem~\ref{theorem:general_theorem}), for the expected, squared approximation error to reach a given tolerance $\tau > \tau_0$, we require
\begin{equation*}
   K \geq \frac{ \log \left(\frac{\tau - \tau_0}{\| x_0 - x_{{\rm LS}} \|^2} \right)}{\log \left(1-1/\Tilde{R}\right)} 
\end{equation*}
iterations.

Let us take a toy example. If we consider $\Sigma= \diag(3~3~1)$, $\|x_{\rm LS}\|=1$, and $\| x_0 - x_{{\rm LS}} \|^2 \approx 10^6$, RK on noiseless system $Ax=b$ needs almost $K \geq 269$ iterations to reach an accuracy of $\epsilon=0.5$, while RK on noisy system $\Tilde{A}x=b$ needs only $K \geq 40$ to reach the same accuracy.

\section{Numerical Results} \label{sec:experiments}
In this section, we report our numerical results on the convergence of the RK algorithm on doubly-noisy linear systems, where both the coefficient matrix, $A$, and the measurement vector, $b$, {have} noise. We empirically validate our theoretical results and compare our bounds of Theorems  \ref{theorem:ls_rk_pnls} and \ref{theorem:general_theorem} in the additive case and evaluate the bound of Corollary \ref{theorem:general_theorem_multiplicative} in the multiplicative case under various noise magnitudes.
\begin{figure}
    \centering
    \includegraphics[width=\textwidth]{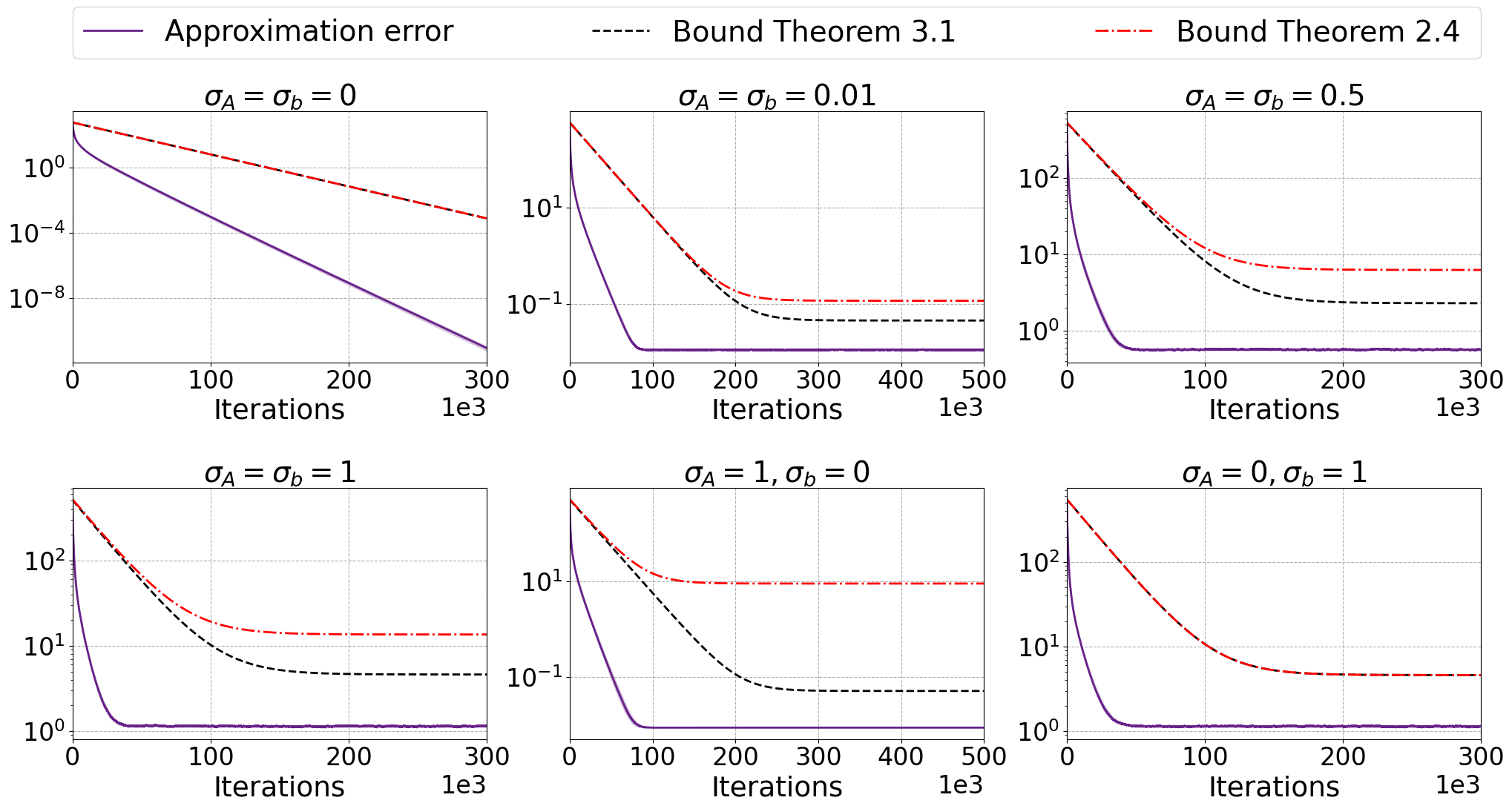}
    \caption{Approximation error, $\|x_k - x_{{\rm LS}}\|$, and the theoretical bounds of Theorems \ref{theorem:ls_rk_pnls} and \ref{theorem:general_theorem} for RK on a noisy linear system. $Ax=b$ and $\Tilde{A}x=b$ are consistent with $m=500$, $n=300$, $r=300$, $\kappa_A=10$, $\sigma_{r}(A)=5$, and $\sigma_{1}(A)=50$. The entries of $\epsilon$ are generated from the standard Normal distribution.}
    \label{fig:comp_3_4}
\end{figure}
\begin{table}
\small
    \centering
    \begin{tabular}{c c c c c c }
    \hline
       &     &     &              & Theoretical convergence   &   Empirical  convergence \\
    $\sigma_A$  &   $\sigma_b$ &   $\kappa_{\Tilde{A}}$ &  $\Tilde{R}$ &         horizon           &         horizon      \\
    \hline
    $0$  &   $0$ &   $10$ &  $11113.545$ & $0$  &   $0$  \\
    \hline
    $0$  &   $1$ &   $10$ &  $11113.545$ & $514.925$  &   $38.951$ \\
    \hline
    $0.005$  &   $0.005$ &   $9.984$ &  $11100.081$ & $4.276$  &   $0.328$ \\
    \hline
    $0.01$  &   $0.01$ &   $9.961$ &  $11052.639$ & $17.011$  &   $1.3$   \\
    \hline
    $0.05$  &   $0.05$ &   $9.759 $ &  $10342.253$ & $385.304 $  &   $30.795$  \\
    \hline
    $0.1$  &   $0.1$ &   $9.792$ &  $10060.141$ & $1365.144 $  &   $97.023$  \\
    \hline
    $0.5$  &   $0.5$ &   $7.35$ &  $5331.502$ & $4704.919$  &   $284.78$  \\
    \hline
    $1$  &   $1$ &   $7.682$ &  $5796.306$ & $6182.523$  &   $310.019$   \\
    \hline
    $1$  &   $0$ &   $7.682$ &  $5796.306$ & $6130.259$  &   $310.388$   \\
    \hline
    $20$  &   $20$ &   $7.264 $ &  $5114.528 $ & $5866.31 $  &   $312.153$   \\
    \hline
    \end{tabular}
    \caption{Effect of the noise on different quantities.}
    \label{tab:comp_exp_additive}
\end{table}
\paragraph{Construction of Consistent Linear Systems} 
We generate the elements of the coefficient matrix, $A$, as follows: for a given dimension, $(m, n)$, and rank $r =\text{rank}(A)$, we choose $\sigma_1(A)$ and $\sigma_{r}(A)$ as the maximum and the minimum nonzero singular values of $A$. We construct $A$ as $A=U \Sigma V^\top$, where $U \in \R^{m\times r}$ and $V \in \R^{n\times r}$. The entries of $U$ and $V$ are generated from the standard Gaussian distribution, and then, the columns are orthonormalized. The matrix, $\Sigma$ is an $r \times r$ diagonal matrix whose diagonal entries are distinct numbers (evenly or randomly spaced) chosen over the interval $[\sigma_{r}(A), \sigma_1(A)]$. As long as the singular values are distinct, the empirical results and conclusions about the behavior of the RK in the presence of noise are similar. {To construct a consistent linear system, we set the measurement vector, $b\in {\rm range}(A).$ 
Note that $x_{\rm LS}=A^\dagger b$.}

\paragraph{Construction of Doubly-Noisy Linear Systems} Given $A$ and $b$ from the consistent linear system constructed above, we construct the noisy data, $\Tilde{A}$ and $\Tilde{b}$, as: \myNum{i} $\Tilde{A} = A + \sigma_A E$, in the additive noise case, and \myNum{ii} $\Tilde{A} = (I+\sigma_A E) A (I+\sigma_A F)$, in the multiplicative noise case. We construct $\Tilde{b} = b + \sigma_b \epsilon$. Note that $\sigma_A, \sigma_b\ge 0$ are the noise magnitudes. To evaluate the bound of Theorem \ref{theorem:ls_rk_pnls}, we construct a consistent partially noisy linear system, $\Tilde{A}x=b$, using noise $E$ that satisfies the condition, $\|A^\dagger\|\|E\| < 1$.

In all the experiments, we average the performance of RK over $10$ trials for each $(\sigma_A, \sigma_b)$ pair. The starting point, $x_0\in {\rm range}(\Tilde{A}^\top)$, and in each experiment, we start from the same $x_0$. The shaded regions in the graphs are given by $\mu \pm 0.5\sigma,$ where $\mu$ is the mean and $\sigma>0$ is the standard deviation of the approximation error. In Section \ref{sec:additive_noise_experiments}, we report the results for the additive noise. In Section \ref{sec:multiplicative_noise_experiments}, we show the results for {the} multiplicative noise. Finally, in Section \ref{sec:preconditoner_experiments}, we present some initial empirical results to validate our remark regarding RK and additive preconditioners; see Section~\ref{sec:preconditioner}.

\begin{figure}
    \centering
    \includegraphics[width=\textwidth]{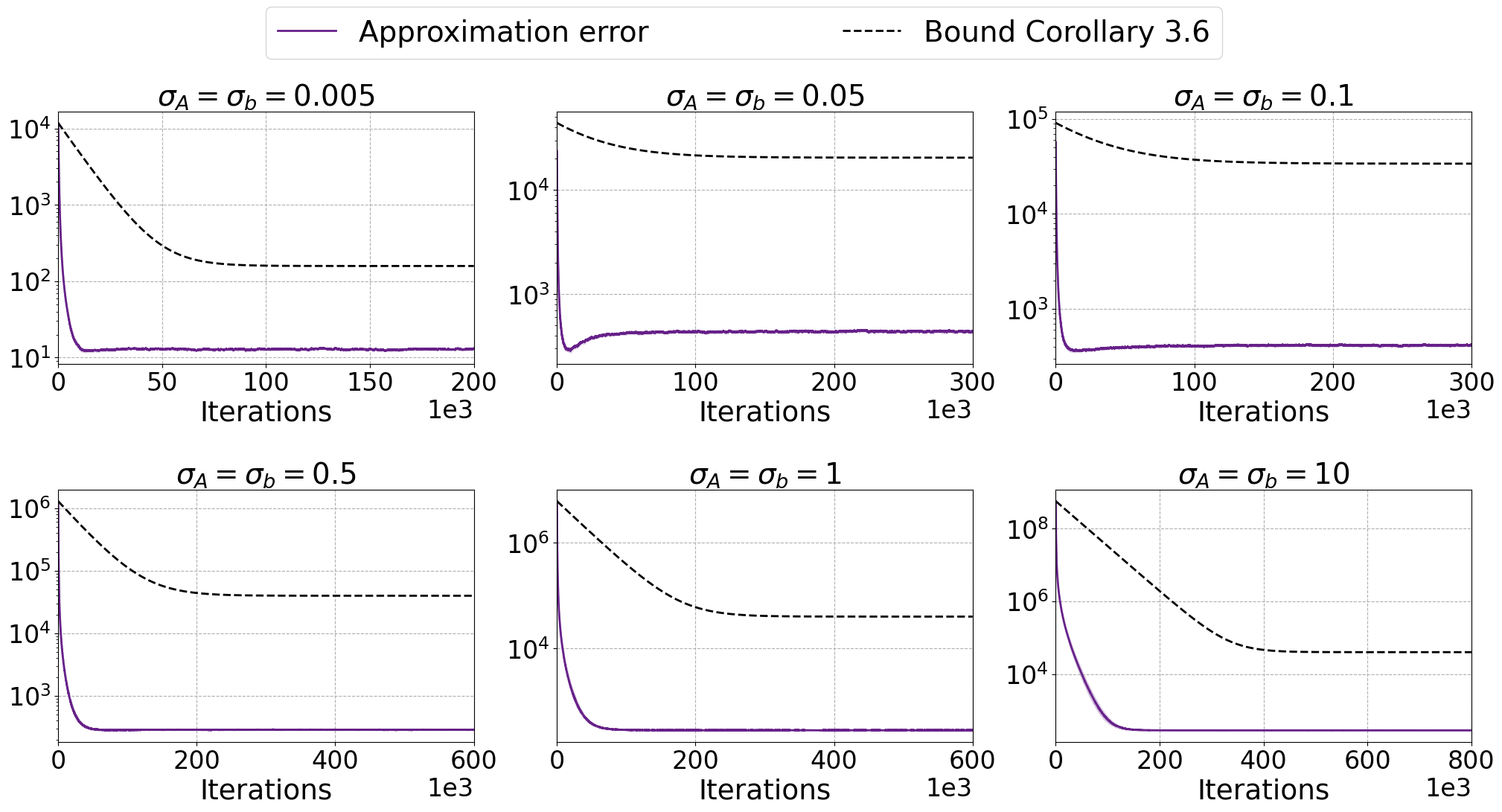}
    \caption{Approximation error, $\|x_k - x_{{\rm LS}}\|^2$, and the theoretical bound of Corollary \ref{theorem:general_theorem_multiplicative} for RK on a noisy linear system. $Ax=b$ is consistent with $m=500$, $n=300$, $r=300$, $\kappa_A=10$, $\sigma_{r}(A)=1$, and $\sigma_{1}(A)=10$. The entries of $E$ and $\epsilon$ are generated from the standard normal distribution, $F=0$.}
    \label{fig:bound6F}
\end{figure}

\subsection{Additive noise} \label{sec:additive_noise_experiments}
Figure \ref{fig:comp_3_4} shows the average empirical approximation error, $\| x_{k} - x_{{\rm LS}}\|$ across iterations and the theoretical bound provided in Theorem \ref{theorem:general_theorem} for different $\sigma_A$ and $\sigma_b$. With added noise, RK converges to within a neighborhood of $x_{{\rm LS}}$, the solution of the noiseless linear system, $Ax=b$. The larger the noise, the further the RK iterates from $x_{{\rm LS}}$. That is, $\| x_{K} - x_{{\rm LS}}\|$ increases as the noise magnitudes ($\sigma_A,\sigma_b$) increase. E.g., for $\sigma_A = \sigma_b = 0.01$, $\| x_{K} - x_{{\rm LS}}\|$ is below $0.1$, while for $\sigma_A = \sigma_b = 0.5$, $\| x_{K} - x_{{\rm LS}}\|$ is around $1$. Increasing the noise also increases the convergence horizon, and the theoretical bound becomes less sharp. Note that in the noise-free case ($\sigma_A=\sigma_b=0$), we validate the existing results about the convergence of RK to $x_{\rm LS}$. 


Table \ref{tab:comp_exp_additive} shows the effect of changing noise magnitudes on the following quantities: The condition number of $\Tilde{A}$ denoted by $\kappa_{\Tilde{A}}$, the scaled condition number, $\Tilde{R}$, of the noisy matrix, the theoretical convergence horizon, $\frac{\| E x_{\rm LS} - \epsilon \|^2}{\sigma_{\rm min}^2(\Tilde{A})}$, and the empirical convergence horizon computed experimentally by  $\mathbb{E}\| x_{K} - x_{{\rm LS}}\|^2$ at the last iteration $K=3\times 10^5$. \footnote{The  empirical convergence horizon is given by $\mathbb{E}\| x_{K} - x_{{\rm LS}}\|^2 - \left(1-\frac{1}{\Tilde{R}}\right)^{K} \| x_0 - x_{{\rm LS}} \|^2$. However, at iteration $K=3\times 10^5$, $\left(1-\frac{1}{\Tilde{R}}\right)^{K} \| x_0 - x_{{\rm LS}} \|^2 \approx 0.$ } The results in the table show that increasing the noise {decreases} the condition number of the coefficient matrix and the value of $\Tilde{R}$, which results in a faster convergence rate. However, this is at the cost of accuracy as both the theoretical and empirical convergence horizons grow as the noise levels increase.

In addition, in Figure \ref{fig:comp_3_4}, we compare the bounds of Theorems \ref{theorem:ls_rk_pnls} and \ref{theorem:general_theorem} in the case of noise $\sigma_A E$ satisfying assumptions of Theorem \ref{theorem:ls_rk_pnls}. We can see that the bound of Theorem \ref{theorem:general_theorem} is better than that of Theorem \ref{theorem:ls_rk_pnls}, which validates our theoretical results stated in Remark~\ref{rem:better_horizon}. 



\subsection{Multiplicative Noise} \label{sec:multiplicative_noise_experiments}
Figures \ref{fig:bound6F}--\ref{fig:bound6} show the approximation error, $\| x_{k} - x_{{\rm LS}}\|^2$ vs. iterations alongside the theoretical bound provided in Corollary \ref{theorem:general_theorem_multiplicative} for various values of $\sigma_A$ and $\sigma_b$. Figures \ref{fig:bound6F} and \ref{fig:bound6E} show the results when $F=0$ and $E=0$, respectively; Figure \ref{fig:bound6} shows the results for the general case, $F\neq 0$ and $E \neq 0$. In all cases, RK converges to a vector that is in a neighborhood of $x_{{\rm LS}}$, similar to the additive noise case, and the radius of the neighborhood increases as the noise magnitudes $(\sigma_A,\sigma_b)$ increase. The curves of the theoretical bound from the figures show that increasing the noise affects the theoretical convergence horizon, and that is particularly due to the right-hand noise, $F$.
\begin{figure}
    \centering
    \includegraphics[width=\textwidth]{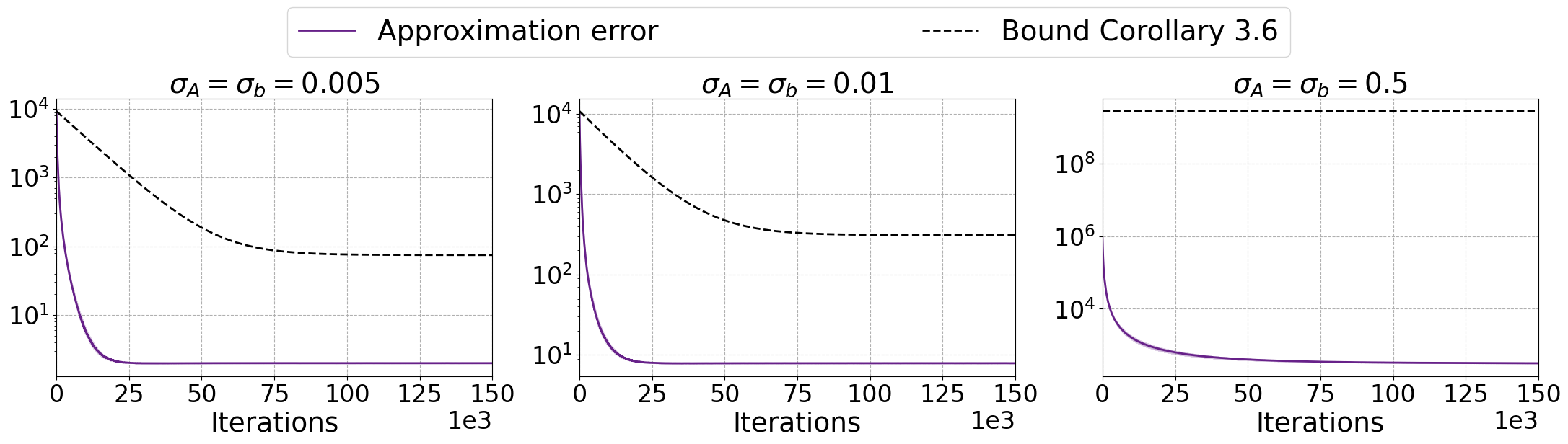}
    \caption{Approximation error, $\|x_k - x_{{\rm LS}}\|^2$, and the theoretical bound of Corollary \ref{theorem:general_theorem_multiplicative} for RK on a noisy linear system. $Ax=b$ is consistent with $m=500$, $n=300$, $r=300$, $\kappa_A=10$, $\sigma_{r}(A)=1$, and $\sigma_{1}(A)=10$. The entries of $F$ and $\epsilon$ are generated from the standard Normal distribution, and $E=0$.}
    \label{fig:bound6E}
\end{figure}

\begin{figure}
    \centering
    \includegraphics[width=\textwidth]{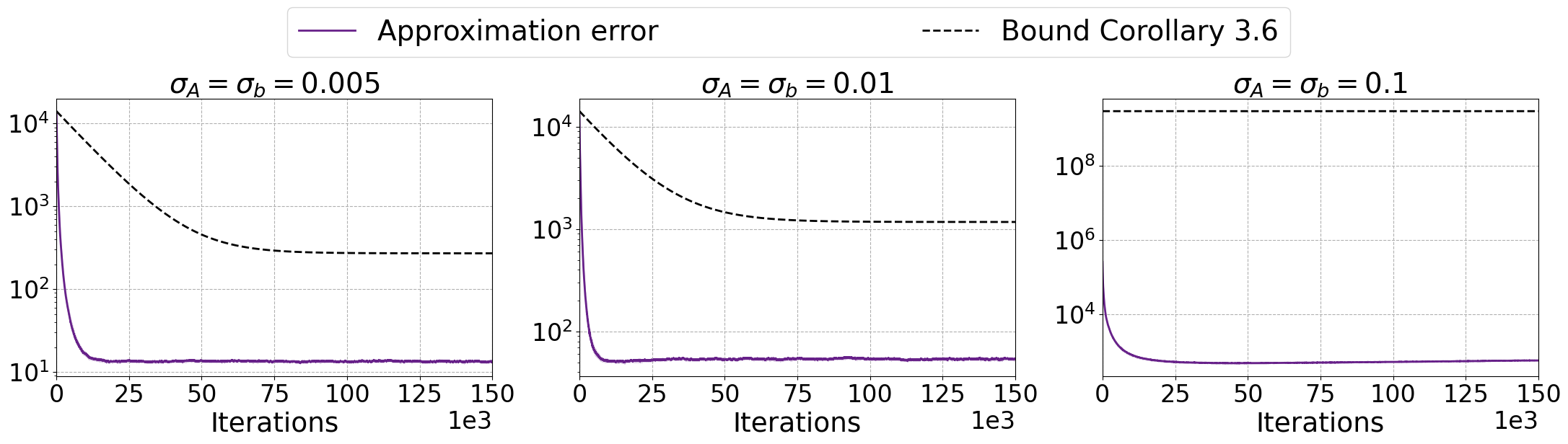}
    \caption{Approximation error, $\|x_k - x_{{\rm LS}}\|^2$, and the theoretical bound of Corollary \ref{theorem:general_theorem_multiplicative} for RK on a noisy linear system. $Ax=b$ is consistent with $m=500$, $n=300$, $r=300$, $\kappa_A=10$, $\sigma_{r}(A)=1$, and $\sigma_{1}(A)=10$. The entries of $F$, $E$, and $\epsilon$ are generated from the standard Normal distribution.}
    \label{fig:bound6}
\end{figure}

\subsection{Additive Preconditioner}\label{sec:preconditoner_experiments}

Figure \ref{fig:preconditioner} shows preliminary numerical results of the additive preconditioner described in Section \ref{sec:preconditioner}. We design the noise, $E$ as explained in Section \ref{sec:preconditioner}. RK was applied to the noise-free system and the noisy system. We observe that at the beginning of the optimization process,  RK applied on the noisy system, $\Tilde{A} x=b$ rapidly decreases the approximation error compared to the noise-free system. 
However, noise-free RK continues to minimize the approximation error while noisy RK becomes stationary. This phenomenon illustrates the benefit of adding a well-crafted noise to the matrix $A$ to speed up the convergence at the beginning of the optimization process. Then one needs to switch back to the noise-free system to continue decreasing the approximation error, $\|x_k - x_{{\rm LS}}\|^2$. 

\begin{figure}
    \centering
     \includegraphics[width=0.48\textwidth]{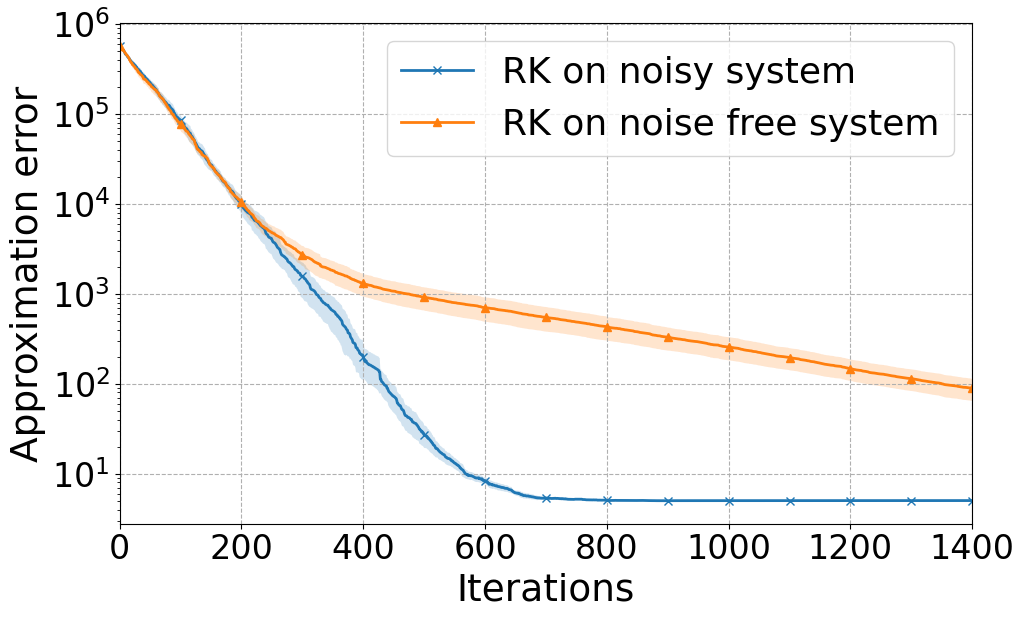}
     \includegraphics[width=0.48\textwidth]{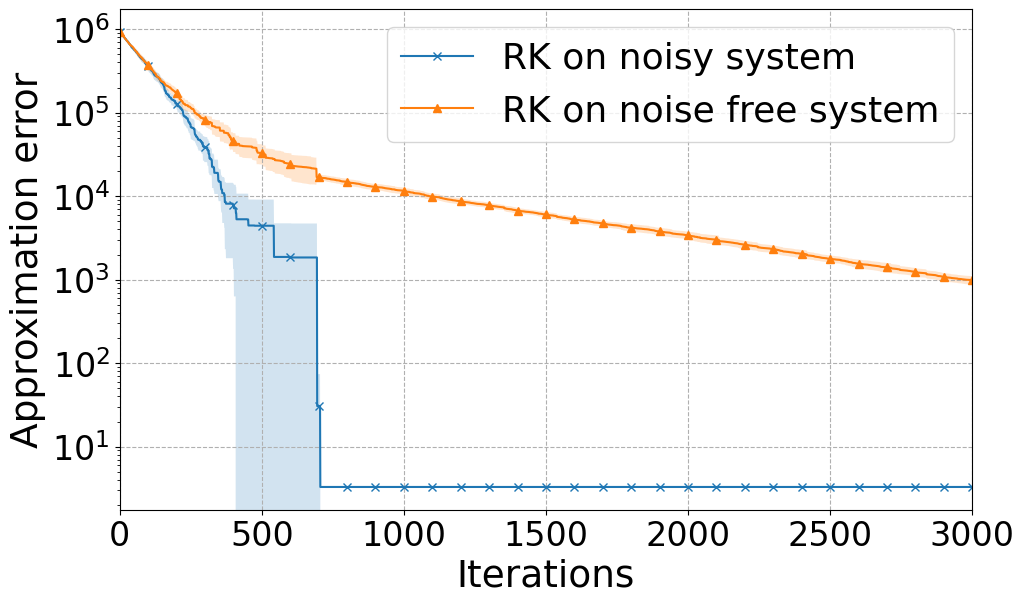}
    \caption{Approximation error, $\|x_k - x_{{\rm LS}}\|^2$, for RK on noise-free system $Ax=b$, and on noisy system $\Tilde{A}x=b$. On the left: $m=100$, $n=50$, $r=50$, $R=785$ and $\Tilde{R}=50$. On the right $m=100$, $n=100$, $r=100$, $R=1585$ and $\Tilde{R}=100$. For both cases $\sigma_{r}(A)=1$, and $\sigma_{i}(A)=4$ for $i=1,...,r-1$.}
    \label{fig:preconditioner}
\end{figure}




\clearpage

\section*{Acknowledgments}
The authors thank the anonymous referees whose feedback significantly improved the quality of this manuscript. Aritra Dutta acknowledges being an affiliated researcher at the Pioneer
Centre for AI, Denmark. 

\bibliographystyle{siamplain}
\bibliography{references}
\end{document}